\documentclass[12pt,a4paper,reqno]{amsart}
\pdfoutput=1

\title{Connectoids I: a universal end space theory}

\author{Nathan Bowler
			\and
			Florian Reich}
\address{Universit\"at Hamburg, Department of Mathematics, Bundesstrasse 55 (Geomatikum), 20146 Hamburg, Germany}
\email{\{nathan.bowler, florian.reich\}@uni-hamburg.de}


\usepackage{amsmath,amssymb,amsthm}
\usepackage{mathtools}
\usepackage{thm-restate}
\usepackage{tikz}
\usetikzlibrary{fadings}
\usepackage{comment}
\usepackage{enumitem}

\usepackage{xcolor}
\usepackage[unicode]{hyperref}
\hypersetup{
	colorlinks,
    linkcolor={red!60!black},
    citecolor={green!60!black},
    urlcolor={blue!60!black},
}
\usepackage[abbrev, msc-links]{amsrefs}
\usepackage[capitalize, nameinlink]{cleveref}

\usepackage[utf8]{inputenc}
\usepackage[T1]{fontenc}
\usepackage{lmodern}
\usepackage[babel]{microtype}
\usepackage[english]{babel}

\usepackage{xcolor}
\definecolor{CornflowerBlue}{rgb}{0.39, 0.58, 0.93}

\usepackage{geometry}
\usepackage[presets={vec-cev}]{letterswitharrows}

\usepackage{enumitem}

\let\polishlcross=\l
\def\l{\ifmmode\ell\else\polishlcross\fi}

\let\emptyset=\varnothing

\let\theta=\vartheta
\let\rho=\varrho
\let\phi=\varphi

\def\NN{\mathbb N}

\def\cC{{\mathcal C}}

\def\cP{{\mathcal P}}

\def\cK{{\mathcal K}}

\newcommand{\Set}[1]{{\left\lbrace {#1} \right\rbrace}}

\def\set#1:#2{\Set{{#1} \colon {#2}}}

\newcommand{\Up}[1]{\lfloor #1 \rfloor}
\newcommand{\OUp}[1]{\mathring{\lfloor #1 \rfloor}}
\newcommand{\Down}[1]{\lceil #1 \rceil}
\newcommand{\ODown}[1]{\mathring{\lceil #1 \rceil}}

\theoremstyle{plain}
\newtheorem{thm}{Theorem}[section]

\newtheorem{prop}[thm]{Proposition}
\newtheorem{cor}[thm]{Corollary}
\newtheorem{lemma}[thm]{Lemma}

\newtheorem{problem}[thm]{Problem}

\theoremstyle{definition}

\begin{document}

\begin{abstract}    
    In this series we introduce and investigate the concept of \emph{connectoids}, which captures the connectivity structure of various discrete objects like undirected graphs, directed graphs, bidirected graphs, hypergraphs or finitary matroids.
    
    In this paper we develop a universal end space theory based on connectoids: the end spaces of connectoids unify the existing end spaces of undirected and directed graphs and establish end spaces for bidirected graphs, hypergraphs and finitary matroids.
    
    The main result shows that the tangle-like description of ends in undirected graphs, called \emph{directions}, extends to connectoids: there is a one-to-one correspondence between the ``directions'' of a connectoid and its ends.
    Furthermore, we generalise normal trees of undirected graphs to connectoids and show that normal trees represent the ends of a connectoid as they do for undirected graphs.
\end{abstract}

\maketitle

\section{Introduction}
Ends of undirected graphs are one of the most important concepts in infinite graph theory.
Halin gave the common definition of ends~\cite{halin1964unendliche}: An \emph{end} of an undirected graph is an equivalence class of rays that are inseparable by deletion of finite vertex sets.
More informally, one can think of an end as a limit point to which rays converge.
Diestel showed \cite{diestel2017ends} that ends are a special case of the infinite order tangles, as introduced by Robertson and Seymour \cite{robertson1991graph}.
Ends have several notable applications, such as for Diestel and K\"uhn's compactification of locally finite graphs \cite{diestel2004infinite}, the existence of infinite grid minors \cite{halin1965maximalzahl},
the study of ubiquity \cites{bowler2018ubiquity, bowler2018ubiquity2, bowler2020ubiquity, bowler2015edge} or the characterisation of vertex-transitive graphs  \cite{mohar1991some}*{Proposition~6.1}.

For directed graphs B\"urger and Melcher recently introduced a notion of ends \cite{burger2020ends}.
Their seminal results \cites{burger2020ends, burger2020ends2, burger2020ends3} show that ends also play a fundamental role in the directed setting.

In this paper we develop a universal end space theory.
This theory unifies the existing end spaces of undirected graphs and directed graphs under omitting limit edges.
Furthermore, this universal end space theory provides a unified notion of ends for a variety of discrete objects that lacked a notion of ends and establishes in this way a fundamental tool for the investigation of these objects.

The basis for this universal end space theory is an abstract description of connectivity:
We consider a set $S$ and a set $\mathcal{F}$ of finite subsets of $S$ such that
\begin{enumerate}[label=(\roman*)]
    \item\label{itm:bonding_closed} for every $F, F' \in \mathcal{F}$ with $F \cap F' \neq \emptyset$ also $F \cup F'$ is an element of $\mathcal{F}$, and
    \item\label{itm:contains_singletons} $\{s\} \in \mathcal{F}$ for every $s \in S$ and $\emptyset \in \mathcal{F}$.
\end{enumerate}
We call a subset $C \subseteq S$ \emph{connected} if for every two elements $x, y \in C$ there is $F \in \mathcal{F}$ with $F \subseteq C$ and $x, y \in F$.
Let $\cC$ be the set of connected sets and call the tuple $(S, \cC)$ the \emph{connectoid induced by $\mathcal{F}$}.
Note that $\mathcal{F}$ is the set of finite connected sets.

Connectoids represent the connectivity structure of common types of graphs and the connectivity structure of finitary matroids.
For the former, let $G$ be either an undirected, directed or bidirected graph, or a hypergraph with finite edges.
Let $S:= V(G)$.
Further, let $\mathcal{F}$ be the set of vertex sets of all finite connected subgraphs of $G$, where connected refers to connectivity for undirected graphs and hypergraphs, and refers to strong connectivity for directed and bidirected graphs.
Thus in the connectoid induced by $\mathcal{F}$ a subset of $S$ is connected if and only if it is the vertex set of a (strongly) connected subgraph of $G$.
Moreover, for displaying the edge-connectivity of $G$, we set $S:= E(G)$ and let $\mathcal{F}$ be the set of edge-sets of all finite connected subgraphs of $G$.

For the latter, let $(E, \mathcal{I})$ be a finitary matroid.
Further, let $S:= E$ and let $\mathcal{F}$ be the set of ground sets of all finite connected restrictions of $(E, \mathcal{I})$. Then the connected sets of the connectoid induced by $\mathcal{F}$ are exactly the ground sets of all connected restrictions of $(E, \mathcal{I})$.

Finitary matroids contain two special cases:
An undirected graph is $2$-connected if and only if its edge set is connected in the matroid sense.
Thus the $2$-connectivity of undirected graphs is reflected by connectoids.
Furthermore, the dual connectivity structure of a graph-like continuum can be expressed in terms of connectoids since it induces a finitary matroid (see \cite{bowler2018infinite} for details).

The concept of connectoids can be considered from different perspectives.
Given a connectoid $(S, \cC)$ and a subset $S' \subseteq S$ we call a maximal connected subset of $S'$ a \emph{component} of $S'$.
This definition generalises the definition of components for all mentioned applications of connectoids.
The set $\cK(S')$ of components of $S'$ partitions $S'$:
We consider the relation $\sim_{S'}$, where $x \sim_{S'} y$ holds if there exists a finite element $F \in \mathcal{C}$ with $F \subseteq S'$ and $x, y \in F$.
Then $\sim_{S'}$ is an equivalence relation since it is reflexive by \labelcref{itm:contains_singletons}, symmetric by the definition and transitive by~\labelcref{itm:bonding_closed}.
This implies that the maximal connected subsets of $S'$ are the equivalence classes under the relation $\sim_{S'}$, which proves that the components of $S'$ partition the set $S'$.

The family of components $(\cK(S'))_{S' \subseteq S}$ characterises $(S, \cC)$:
Note that
\begin{itemize}
    \item $(\cK(S'))_{S' \subseteq S}$ is \emph{finitary}, i.e.\ for every $K \in \bigcup_{S' \subseteq S} \cK(S')$ and every $x, y \in K$ there is a finite element $K' \in \bigcup_{S' \subseteq S} \cK(S')$ such that $x, y \in K' \subseteq K$, and
    \item $(\cK(S'))_{S' \subseteq S}$ is \emph{consistent}, i.e.\ for every $\hat S \subseteq S' \subseteq S$ the partition $\cK(\hat{S})$ refines the partition $\cK(S')$ and $\{ K \in \cK(S'): K \subseteq \hat S\} \subseteq \cK( \hat S)$.
\end{itemize}
Conversely, every finitary consistent family $(\cP(S'))_{S' \subseteq S}$ of partitions induces a connectoid: The set $\Tilde{\mathcal{F}}:= \bigcup_{S' \in S^{< \infty}} \cP(S')$ satisfies \labelcref{itm:bonding_closed} by consistency and \labelcref{itm:contains_singletons} since singletons are the unique partition classes of one element subsets of $S$.
Then in the connectoid induced by $\Tilde{\mathcal{F}}$ the components of $S'$ are exactly the elements of $\cP(S')$ for every $S' \subseteq S$ by finitarity of $(\cP(S'))_{S' \subseteq S}$.

Connectoids can also be characterised as hypergraphs with finite edges:
On the one hand, a connectoid $(S, \cC)$ generates the hypergraph with vertex set $S$ and edge set $\mathcal{F}$, i.e.\ every finite connected set forms an edge.
On the other hand, the finite connected sets of a hypergraph satisfy \labelcref{itm:bonding_closed} and \labelcref{itm:contains_singletons} (see \cite{dewar2018connectivity} for more details on connectivity of hypergraphs) and therefore induce a connectoid.
In particular, the hypergraph with vertex set $S$ and edge set $\mathcal{F}$ induces $(S, \cC)$.

Before turning our attention to the definition of ends we introduce a counterpart to the graph-theoretic notion of rays, which is inspired by B\"urger and Melcher's~\cite{burger2020ends} necklaces of directed graphs. A connected set $N \in \cC$ is called a \emph{necklace} if there exists a family $(H_n)_{n \in \NN}$ of finite connected sets such that $N = \bigcup_{n \in \NN} H_n$ and $H_i \cap H_j \neq \emptyset$ holds if and only if $|i - j| \leq 1$ for every $i, j \in \NN$. We call the family $(H_n)_{n \in \NN}$ a \emph{witness} of $N$.

Given a necklace $N$ and a finite set $X \subseteq S$ there exists a component in $\cK(N \setminus X)$ containing \emph{almost all}, i.e.\ all but finitely many, elements of $N$: Let $(H_n)_{n \in \NN}$ be a witness of $N$. Then there exists $m \in \NN$ such that $H_n \cap X = \emptyset$ for every $n \geq m$. Therefore there is a component in $\cK(N \setminus X)$ that contains $\bigcup_{n \geq m} H_n$ and thus almost all elements of $N$.
We call this unique component the \emph{$X$-tail} of $N$ and prove that it is again a necklace.

Two necklaces $N$ and $N'$ are \emph{equivalent} if for every finite set $X \subseteq S$ the $X$-tails of $N$ and $N'$ are contained in the same component in $\cK( S \setminus X)$.
An equivalence class of necklaces under this relation is an \emph{end} of $(S, \cC)$.
We define \emph{$\Omega(S, \cC)$} to be the set of ends of $(S, \cC)$.
Given an end $\omega \in \Omega(S, \cC)$ and a finite set $X \subseteq S$ we let \emph{$K(X, \omega)$} be the component containing the $X$-tails of all necklaces in $\omega$.
The topology of the \emph{end space} of $(S, \cC)$ is generated by the basic open sets of the form $\{\omega \in \Omega(S, \cC): K(X, \omega) = K\}$ for all finite sets $X \subseteq S$ and all components $K$ in $\cK(S \setminus X)$.

We show that this notion of ends of connectoids is reasonable:
Firstly, we prove in \cref{sec:relation_ends} that the end spaces of undirected and directed graphs are homeomorphic to the end spaces of the corresponding connectoids. For undirected graphs this homeomorphism can be defined by the observation that the vertex sets of equivalent rays form equivalent necklaces.

Secondly, as for undirected and directed graphs~\cite{diestel2003graph}*{Theorem~2.2}, ends of connectoids correspond to directions: a \emph{direction} of a connectoid $(S, \cC)$ is a map $d$ that sends every finite set $A \subseteq S$ to a component in $\cK(S \setminus A)$ such that $d(B) \subseteq d(A)$ holds for every $A \subseteq B \in S^{< \infty}$. We prove
\begin{restatable}{thm}{DirectionTheorem} \label{thm:direction}
    Let $(S, \cC)$ be a connectoid. For every end $\omega$ of $(S, \cC)$ there is a unique direction $d_\omega$ such that $d_\omega(A) = K(A, \omega)$ for every finite set $A \subseteq S$. The map sending $\omega$ to $d_\omega$ is a bijection between the ends of $(S, \cC)$ and the directions of $(S, \cC)$.
\end{restatable}

Thirdly, we show that ends of connectoids have a close relation to a notion of normal trees.
A \emph{weak normal tree} $T$ of a connectoid $(S, \cC)$ is a rooted, undirected tree $T$ with $V(T) \subseteq S$ such that
\begin{itemize}
    \item for every connected set $C \in \cC$ and every two $\leq_T$-incomparable elements $u, v \in C \cap V(T)$ there exists $w \in C$ such that $w \leq_T u, v$, and
    \item for every two $\leq_T$-comparable elements $u \leq_T v$ there exists a connected set $C \in \cC$ containing $u$ and $v$ that avoids every element $w <_T u$,
\end{itemize}
where $\leq_T$ refers to the tree-order of $T$.
If additionally for every rooted ray $R$ in $T$ there is a necklace that contains almost all elements of $V(R)$ we call $T$ a \emph{normal tree} of $(S, \cC)$.
Furthermore, $T$ is \emph{spanning}  if $V(T) = S$.
We will show later that every weak normal tree that is spanning is normal.

The definition of normal trees of connectoids differs from that of normal trees in undirected graphs in the fact that the edges of the tree are not a substructure of the connectoid since connectoids do not contain edges.
But for an undirected connected graph $G$ and a subset $U \subseteq V(G)$ there is a normal tree containing $U$ if and only if there is a normal tree $T$ in the corresponding connectoid that contains $U$:
For the forward direction the tree $T$ is a normal tree in the corresponding connectoid since its rooted rays form necklaces.
For the backward direction note that $T$ is a normal tree covering $U$ in $G \cup T$.
Since the existence of normal trees is closed under taking connected subgraphs~\cite{pitz2021proof}*{Theorem~1.2}, there exists a normal tree containing $U$ in $G$.

We show that the ends of a normal spanning tree reflect the ends of its connectoid:
\begin{restatable}{thm}{EndFaithful}\label{thm:endfaithfullness}
Let $(S, \cC)$ be a connectoid with a normal spanning tree $T$. Then there is a bijection between the ends of $(S, \cC)$ and the ends of $T$.
\end{restatable}
\noindent
Further, we prove that the ends of every connected connectoid can be approximated by rayless normal trees:
\begin{restatable}{thm}{Approximate}\label{thm:approx} 
Let $(S, \cC)$ be a connected connectoid. For every collection $\mathcal{K} = \{K(X_\omega, \omega): \omega \in \Omega(S, \cC)\}$ there is a rayless normal tree $T$
of $(S, \cC)$ such that every component in $\cK(S \setminus V(T))$ is subset of an element of $\mathcal{K}$.
\end{restatable}

\noindent
\cref{thm:approx} generalises Kurkofka, Melcher and Pitz' result \cite{kurkofka2021approximating}*{Theorem~1} for undirected graphs.
It enables us to show that the end spaces of connectoids are \emph{ultra-paracompact}, i.e.\ every open cover of an end space can be refined to an open partition. 
Furthermore, using \cref{thm:approx} we will show in the second paper of this series that a normal spanning tree exists if normal spanning trees exist locally at each end.

In the second paper of this series~\cite{connectoids2} we will study normal trees of connectoids by characterising their existence.
This includes a Jung-type \cite{jung1969wurzelbaume}*{Theorem~6} characterisation via dispersed sets.
Furthermore, we show that the existence of a tree satisfying the first condition of weak normal trees suffices for the existence of normal trees.
In this way we show that this notion of normal tree is sensible and establish normal trees for (bi)directed graphs, hypergraphs and finitary matroids.
The second author will investigate this noval concept of normal trees in the setting of directed graphs in~\cite{normaltreesofdigraphs}.

\medskip

This paper is organised as follows. We discuss basic properties of necklaces and weak normal trees in \cref{sec:prelims}.
In \cref{sec:direction_thm} we prove the direction theorem, \cref{thm:direction}, and in \cref{sec:relation_ends} we show that ends of connectoids generalise ends of undirected and directed graphs.
In \cref{sec:normal_tree} we show \cref{thm:endfaithfullness}.
We prove \cref{thm:approx} in \cref{sec:approx} and characterise the class of connectoids that can be compactified in \cref{sec:compactification}.
Finally, we state open problems in \cref{sec:open_problems}.

\section{Preliminaries} \label{sec:prelims}
For standard graph-theoretic notations we refer to Diestel's book \cite{DiestelBook2016}.
Let $\NN:= \{1, 2, 3, \dots \}$ and $\NN_0:= \{0\} \cup \NN$.
Given a partial order $\leq$ of a set $S$ and some element $s \in S$ we define $\Up{s}_\leq:= \{ x \in S: x \geq s \}$ and $\OUp{s}_\leq:= \Up{s}_\leq \backslash \{s\}$.
The sets $\Down{s}_\leq$ and $\ODown{s}_\leq$ are defined analogously.
Given the tree order $\leq_T$ of some rooted tree $T$ we simply write $\Up{s}_T, \OUp{s}_T, \Down{s}_T$ and $\ODown{s}_T$ for the respective sets.
Given a tree $T$, we call every $\leq_T$-down-closed set of pairwise $\leq_T$-comparable elements in $V(T)$ a \emph{branch} of $T$. 
For simplicity, we write $t \in T$ instead of $t \in V(T)$ for a tree $T$.
Given a connectoid $(S, \cC)$, some tree $T$ with $V(T) \subseteq S$ and some $t \in T$, we define $K_t^T$ to be the unique component in $\cK(S \setminus\ODown{t}_T)$ containing $t$.

We call a connectoid $(S, \cC)$ \emph{connected} if $S$ is a connected set.
A connectoid $(S', \cC')$ is a \emph{subconnectoid} of $(S, \cC)$ if $S' \subseteq S$ and $\cC' \subseteq \cC$.
Note that $(S', \cC \cap \cP(S'))$ is a subconnectoid of $(S, \cC)$ for every $S' \subseteq S$, and we refer to it as the \emph{induced subconnectoid} of $(S, \cC)$ on $S'$.

\subsection{Necklaces}
For undirected graphs it is a basic fact that two rays $R, R'$ are equivalent if and only if there are infinitely many disjoint $R$-$R'$ paths~\cite{DiestelBook2016}.
We show:
\begin{prop}\label{prop:necklace_property}
    Let $(S, \cC)$ be a connectoid and $N, N'$ necklaces of $(S, \cC)$.
    Then $N$ and $N'$ are equivalent if and only if there is an infinite family $(C_n)_{n \in \NN}$ of disjoint, finite connected sets such that every $C_n$ contains elements of $N$ and $N'$.
\end{prop}
\begin{proof}
    If $N$ and $N'$ are equivalent, we construct the desired infinite family $(C_n)_{n \in \NN}$ recursively: 
    Assume that $(C_n)_{n \leq m}$ has been defined for some $m \in \NN$.
    Then let $C_{n + 1}$ be some finite connected set in the unique component in $\cK(S \setminus (\bigcup_{n \leq m} C_n))$ that contains the $(\bigcup_{n \leq m} C_n)$-tails of $N$ and $N'$ such that $C_{n + 1}$ contains an element of $N$ and $N'$.
    Then $C_{n + 1}$ is as desired.

    If there is an infinite family $(C_n)_{n \in \NN}$ of finite connected sets such that $C_n$ contains a vertex of $N$ and $N'$ for every $n \in \NN$, then for every finite set $X \subseteq S$ there is $m \in \NN$ such that $C_m$ avoids $X$ and such that it has a vertex in the $X$-tails of $N$ and $N'$.
    Thus the $X$-tails of $N$ and $N'$ are contained in the same component in $\cK(S \setminus X)$.
    This implies that $N$ and $N'$ are equivalent.
\end{proof}

We say an infinite set $X \subseteq S$ \emph{converges} to an end $\omega$ if $K(Y, \omega)$ contains almost all elements of $X$ for every finite set $Y \subseteq S$. Note that an infinite set converges to at most one end.

\begin{prop}\label{prop:converging}
	Let $(S, \cC)$ be a connectoid and $X \subseteq S$ a countably infinite set. Then the following properties are equivalent:
	\begin{enumerate}[label=(\roman*)]
		\item\label{itm:converging_1} there exists an end $\omega \in \Omega(S, \cC)$ such that $X$ converges to $\omega$,
		\item\label{itm:converging_2} for every $Y \in S^{< \infty}$ there is a component in $\cK(S \setminus Y)$ that contains almost all elements of $X$, and
		\item\label{itm:converging_3} there exists a necklace that contains almost all elements of $X$.
	\end{enumerate}
	Furthermore, if $(S, \cC)$ is connected, then \labelcref{itm:converging_1,itm:converging_2,itm:converging_3} hold true if and only if there is a necklace that contains all elements of $X$.
\end{prop}
\begin{proof}
	\begin{description}
		\item[\labelcref{itm:converging_1} implies \labelcref{itm:converging_2}] 
		If $X$ converges to an end $\omega$, then the component $K(Y, \omega)$ is as desired for every finite set $Y \subseteq S$.
		\item[\labelcref{itm:converging_2} implies \labelcref{itm:converging_3}]
		Let $S'$ be the component in $\cK(S)$ that contains almost all elements of $X$.
		We set $X' := S' \cap X$ and let $\{x_n: n \in \NN\}$ be an enumeration of $X'$.
		We build recursively a witness $(H_n)_{n \in \NN}$ of a necklace such that
		\begin{enumerate}[label=(\alph*)]
			\item \label{itm:converging_2_a} the set $X' \setminus \bigcup_{i \leq n} H_i$ is contained in a common component $K_n \in \cK(S \setminus \bigcup_{i < n} H_i)$,
			\item \label{itm:converging_2_b} $H_n$ contains an element of $K_n$, and
			\item \label{itm:converging_2_c} $H_n$ contains the element of $X' \setminus \bigcup_{i < n} H_i$ with least index.
		\end{enumerate}
		Then $\bigcup_{n \in \NN} H_n$ is a necklace that contains $X'$ by \labelcref{itm:converging_2_c}, and thus is as desired.
		
		Let $H_1$ be some finite connected set containing $x_1$.
		We assume that $H_n$ has been constructed for some $n \in \NN$.
		By assumption, there is a unique component $K_{n + 1} \in \cK(S \setminus \bigcup_{i \leq n} H_i)$ that contains almost all elements of $X'$. By \labelcref{itm:converging_2_a}, $K_n$ contains $K_{n + 1}$ and avoids $\bigcup_{i < n} H_i$.
		The connected set $K_n$ contains an element of $H_n$ by \labelcref{itm:converging_2_b} and all elements of $X' \setminus \bigcup_{i \leq n} H_i$ by \labelcref{itm:converging_2_a}.
		Let $H_{n +1}$ be a finite connected set in $K_n$ containing an element of $H_n$, the element of $X' \setminus \bigcup_{i \leq n} H_i$ with least index, an element of $K_{n + 1}$ and all elements of $X' \setminus (\bigcup_{i \leq n} H_i \cup K_{n +1})$.
		Then $(H_i)_{i \leq n+1}$ can indeed be extended to a witness. Furthermore, \labelcref{itm:converging_2_a}, \labelcref{itm:converging_2_b} and \labelcref{itm:converging_2_c} hold by construction of $H_{n + 1}$. This finishes the construction of the witness $(H_n)_{n \in \NN}$.
		\item[\labelcref{itm:converging_3} implies \labelcref{itm:converging_1}]
		If there exists a necklace $N$ that contains almost all elements of $X$, then the $Y$-tail of $N$ contains almost all elements of $X$ for every finite set $Y \subseteq S$ since the $Y$-tail of $N$ contains almost all elements of $N$. Let $\omega \in \Omega(S, \cC)$ be the end that contains $N$. As the $Y$-tail of $N$ is contained in $K(Y, \omega)$, almost all elements of $X$ are contained in $K(Y, \omega)$ for every finite set $Y \subseteq S$.
	\end{description} 
	If $(S, \cC)$ is connected, then $X' =X$ in the proof of `\labelcref{itm:converging_2} implies \labelcref{itm:converging_3}', which provides a necklace containing $X$.
\end{proof}

As tails of rays are rays, tails of necklaces are also necklaces:
\begin{cor}
	Let $N \subseteq S$ be a countably infinite connected set in some connectoid $(S, \cC)$.
	Then $N$ is a necklace if and only if there is a component in $\cK(N \setminus Y)$ that contains almost all elements of $N$ for every finite set $Y \subseteq N$.\\
	Furthermore, the $X$-tail of a necklace $N$ is a necklace for every finite set $X \subseteq S$.
\end{cor}
\begin{proof}
	Note that $N$ is a necklace if and only if there is a necklace in the subconnectoid $(N, \cC \cap \mathcal{P}(N))$ that contains $N$.
	By \cref{prop:converging} and since $(N, \cC \cap \mathcal{P}(N))$ is connected, there is a necklace in $(N, \cC \cap \mathcal{P}(N))$ that contains $N$ if and only if there is a component in $\cK(N \setminus Y)$ that contains almost all elements of $N$ for every finite $Y \subseteq N$. 
	
	For the `furthermore'-part, let $X \subseteq N$ be an arbitrary finite set and let $N'$ be the $X$-tail of $N$.
	For every $Y \subseteq N'$ there is a component $K \in \cK(N \setminus (X \cup Y))$ that contains almost all elements of $N$ since $N$ is a necklace.
	By the definition of $N'$, $K$ is contained in $N'$ and thus a component in $\cK(N' \setminus Y)$. This shows that $N'$ is a necklace.
\end{proof}

\subsection{Weak normal trees}
\begin{prop}\label{prop:equivalence_weak_normal_tree}
    Let $(S, \cC)$ be a connectoid and $T$ a rooted, undirected tree with $V(T) \subseteq S$. Then the following properties are equivalent:
    \begin{enumerate}[label=(\arabic*)]
        \item\label{itm:equivalence_weak_normal_tree_1} $T$ is a weak normal tree, and
        \item\label{itm:equivalence_weak_normal_tree_2} $K_t^T \cap V(T) = \Up{t}_T$ holds for every $t \in T$.
    \end{enumerate}
    Furthermore, $K_t^T = \Up{t}_T$ holds for every vertex $t$ of a weak normal spanning tree $T$.
\end{prop}
\begin{proof}
\begin{description}
    \item[\labelcref{itm:equivalence_weak_normal_tree_1} implies \labelcref{itm:equivalence_weak_normal_tree_2}]
        Let $t \in T$ be an arbitrary vertex. Then $K_t^T \cap V(T)$ is contained in $\Up{t}_T$ since every connected set containing $t$ and some element $v \in T$ that is $\leq_T$-incomparable to $t$ has to contain an element of $\ODown{t}_T$. Furthermore, every element $v \in \Up{t}_T$ is contained in $K_t^T$ as there is a connected set containing $t$ and $v$ that avoids $\ODown{t}_T$. Thus $K_t^T \cap V(T) = \Up{t}_T$ holds for every vertex $t \in T$.
    \item[\labelcref{itm:equivalence_weak_normal_tree_2} implies \labelcref{itm:equivalence_weak_normal_tree_1}]
        Firstly, suppose for a contradiction that there is a connected set $C \in \cC$ containing two $\leq_T$-incomparable elements $u, v$ with $C \cap \ODown{u}_T \cap \ODown{v}_T = \emptyset$. We assume without loss of generality that $u$ is $\leq_T$-minimal with the property $u \in C$.
        Then $C \subseteq K_u^T$ since $C \cap \ODown{u}_T = \emptyset$. But $v \notin K_u^T$ since $K_u^T \cap V(T) = \Up{u}_T$, a contradiction.

        Secondly, for every two $\leq_T$-comparable elements $u \leq_T v$ the connected set $K_u^T$ contains $u$ and $v$ and avoids $\ODown{u}_T$. 
\end{description}
    For the \qq{furthermore}-part we use \labelcref{itm:equivalence_weak_normal_tree_2} together with $V(T) = S$.
\end{proof}

\section{Direction theorem}\label{sec:direction_thm}

\DirectionTheorem*

\begin{proof}
    Let $\omega \in \Omega(S, \cC)$ be arbitrary. The function $d_\omega$ is indeed a direction since $ K(B, \omega) \subseteq K(A, \omega)$ for every $A \subseteq B \in S^{< \infty}$. For $\omega \neq \omega' \in \Omega(S, \cC)$ there exists a finite set $A \subseteq S$ such that $K(A, \omega) \neq K(A, \omega')$ by the definition of end. Thus the map sending $\omega$ to $d_\omega$ is injective.
    It remains to prove that this map is surjective.
    Let $d$ be an arbitrary direction of $(S, \cC)$. We show the existence of an end $\omega \in \Omega(S, \cC)$ such that $d(A) = K(A, \omega)$ for every finite set $A \subseteq S$. 

    We construct a sequence $C_1, C_2, \dots$ of disjoint finite connected sets of $(S, \cC)$ with the properties
    \begin{enumerate}[label=(\roman*)]
        \item\label{itm:direction1} $C_n \subseteq d(\bigcup_{i = 1}^{n - 1} C_i)$,
        \item\label{itm:direction2} $d(X \cup \bigcup_{i = 1}^{n-1} C_i) \cap C_{n} \neq \emptyset$ holds for every finite $X \subseteq d(\bigcup_{i = 1}^n C_i)$
    \end{enumerate}
    for every $n \in \NN$ by recursively adding a suitable finite connected set as long as possible.
	We remark that $\bigcup_{i= 1}^{0} C_i := \emptyset$.
    \begin{description}
        \item[Case 1: the sequence is infinite]
            Set $K_n:= d(\bigcup_{i= 1}^{n - 1} C_i)$ for every $n \in \NN$.
            Then $K_1 \supset K_2 \supset K_3 \supset \dots$ is a strictly decreasing sequence of connected sets with limit $K:= \bigcap_{n \in \NN} K_n$ by property \labelcref{itm:direction1}. We recursively construct a witness $(H_n)_{n \in \NN}$ of a necklace $N$ and a strictly increasing sequence $(m(n))_{n \in \NN_0} \in \NN^{\NN_0}$ such that for every $n \in \NN$
            \begin{enumerate}[label=(\alph*)]
                \item\label{itm:direction_c1_3} $K_{m(n)} \cap \bigcup_{i < n} H_i \subseteq K$,
                \item\label{itm:direction_c1_1} $H_n \subseteq K_{m(n - 1)}$, and
                \item\label{itm:direction_c1_2} $C_{m(n)} \subseteq H_n$.
            \end{enumerate}
            
            Let $\ell \in \NN_0$ and assume that $(H_n)_{n \leq \ell}$ and $(m(n))_{n \leq \ell}$ have been defined. As $\bigcup_{i < \ell + 1} H_i$ is finite, there is $m( \ell + 1) \in \NN$ such that $K_{m(\ell + 1)} \cap \bigcup_{i < \ell + 1} H_i \subseteq K$, which ensures property \labelcref{itm:direction_c1_3}.
            
            Now we consider the set $X:= K_{m(\ell)} \cap \bigcup_{i < \ell} H_i$, which is subset of 
            $K$ by property \labelcref{itm:direction_c1_3}. In particular, $X$ is subset of $K_{m(\ell) + 1}$ and $K_{m(\ell + 1) + 1}$.
            The connected set $d(X \cup \bigcup_{i = 1}^{m(\ell) - 1} C_i)$ intersects $C_{m(\ell)}$ by property \labelcref{itm:direction2}. As $C_{m(\ell)}$ is disjoint to $X \cup \bigcup_{i = 1}^{m(\ell) - 1} C_i$, the set $C_{m(\ell)}$ is even contained in $d(X \cup \bigcup_{i = 1}^{m(\ell) - 1} C_i)$. Similarly, the connected set $d(X \cup \bigcup_{i = 1}^{m(\ell + 1) - 1} C_i) \subseteq d(X \cup \bigcup_{i = 1}^{m(\ell) - 1} C_i)$ contains $C_{m(\ell + 1)}$. Let $H_{\ell + 1}$ be a finite connected set in $d(X \cup \bigcup_{i = 1}^{m(\ell) - 1} C_i)$ containing $C_{m(\ell)}$ and $C_{m(\ell + 1)}$.

            As $H_{\ell + 1}$ contains $C_{m(\ell)}$, it intersects $H_\ell$ by property \labelcref{itm:direction_c1_2}. As $H_{\ell + 1} \subseteq d(X \cup \bigcup_{i = 1}^{m(\ell) - 1} C_i)$, it is contained in $K_{m(\ell)}$ and avoids $X$. This implies that $H_{\ell + 1}$ avoids $H_i$ for every $i < \ell$ by choice of $X$. Therefore, $(H_i)_{i \leq \ell + 1}$ can indeed be an initial segment of a witness of a necklace.
            Furthermore, property \labelcref{itm:direction_c1_1} is ensured and property \labelcref{itm:direction_c1_2} holds by construction. This finishes the construction of the witness $(H_n)_{n \in \NN}$.
            
            It remains to prove that $d(Y)$ contains a tail of $N$ for every finite set $Y \subseteq S$.
            Let $Y \subseteq S$ be an arbitrary finite set and pick $j \in \NN$ such that $Y \cap K_j \subseteq K$ holds. Then $Y$ splits into $Y \setminus K_j$ and $Y \cap K$. Then $d(Y \cup \bigcup_{i = 1}^{j - 1} C_i) = d((Y \cap K) \cup \bigcup_{i = 1}^{j - 1} C_i)$ meets $C_k$ for every $k \geq j$ since $ d((Y \cap K) \cup \bigcup_{i = 1}^{k - 1} C_i) \subseteq d((Y \cap K) \cup \bigcup_{i = 1}^{j - 1} C_i)$ does so by property~\labelcref{itm:direction2}.
            Take $p \in \NN$ such that the tail $\bigcup_{n \geq p} H_n$ of $N$ is disjoint to $Y$. Then the tail $\bigcup_{n \geq p} H_n$ contains $C_{m(n)}$ for every $n \geq p$ and therefore meets $d(Y \cup \bigcup_{i = 1}^{j - 1} C_i)$.
            Since $d(Y \cup \bigcup_{i = 1}^{j - 1} C_i) \subseteq d(Y)$, the tail $\bigcup_{n \geq p} H_n$ is contained in $d(Y)$, which completes the first case.

        \item[Case 2: the sequence is finite]
            Let $k \in \NN$ be the length of the sequence and set $Z:= \bigcup_{i= 1}^{k} C_i$.
            We construct a sequence of disjoint finite connected sets $(B_n)_{n \in \NN}$ in $d(Z)$ such that $d(B_1 \cup Z) \supset d(B_2 \cup Z) \supset d(B_3 \cup Z) \supset \dots$ is a strictly decreasing sequence of connected sets with $\bigcap_{n \in \NN} d(B_n \cup Z) = \emptyset$.

            Set $B_1 := \emptyset$. Let $\ell \in \NN_0$ and assume that $(B_n)_{n \leq \ell}$ has been defined. As $B_\ell \subseteq d(Z)$ has property \labelcref{itm:direction1} but does not extend the sequence $(C_n)_{n \leq k}$, property \labelcref{itm:direction2} does not hold. Thus there exists a finite set $X \subseteq d(B_\ell \cup Z)$ with $d(X \cup Z) \cap B_\ell = \emptyset$. Let $B_{\ell +1}$ be a finite nonempty connected set in $d(B_\ell \cup Z)$ that contains $X$.
            Then $d(B_{\ell + 1} \cup Z) \subseteq d(X \cup Z)$ avoids $B_\ell$ and is therefore contained in $d(B_\ell \cup Z)$. By construction, $B_{\ell + 1} \subseteq d(B_\ell \cup Z)$ holds and therefore $B_{\ell + 1}$ is disjoint to $B_n$ for $n \leq \ell$. This finishes the construction of the sequence $(B_n)_{n \in \NN}$. Note that the sequence $(B_n)_{n \in \NN}$ has the additional property that $B_{n + 1} \subseteq d(Z \cup B_n)$ for every $n \in \NN$.
            
            Suppose for a contradiction that there is $b \in \bigcap_{n \in \NN} d(B_n \cup Z)$ and pick a finite connected set $A \subseteq d(Z)$ containing $b$ and $B_2$. Note that $A$ avoids $Z$. The connected set $A$ contains an element of $B_j$ for all $j \geq 2$, as $b \in A \cap d(B_j \cup Z)$ and $B_2 \subseteq A \setminus d(B_j \cup Z)$, for every $j \in \NN$. This contradicts that $A$ is finite but the elements of $(B_j)_{j \in \NN}$ are disjoint.

            Set $K_1 := d(Z)$ and $K_n := d(B_{n - 1} \cup Z)$ for every $n \geq 2$. We construct a witness $(H_k)_{k \in \NN}$ of a necklace $N$ and a strictly increasing sequence $(m(n))_{n \in \NN_0} \in \NN^{\NN_0}$ such that
            \begin{enumerate}[label=(\alph*)]
                \item\label{itm:direction_c2_1} $K_{m(n)} \cap \bigcup_{i < n} H_i = \emptyset$,
                \item\label{itm:direction_c2_2} $H_n \subseteq K_{m(n - 1)}$, and
                \item\label{itm:direction_c2_3} $B_{m(n)} \subseteq H_n$.
            \end{enumerate}
            Let $\ell \in \NN_0$ and assume that $(H_n)_{n \leq \ell}$ and $(m(n))_{n \leq \ell}$ have been defined. Let $m( \ell + 1) > m(\ell)$ be such that $K_{m( \ell + 1)} \cap \bigcup_{n < \ell + 1} H_n = \emptyset$, which is possible since $\bigcap_{n \in \NN} K_n = \emptyset$. Then property \labelcref{itm:direction_c2_1} holds.
            Let $H_{\ell + 1}$ be some finite connected set in $K_{m(\ell)}$ containing $B_{m(\ell)}$ and $B_{m(\ell + 1)}$. Then properties \labelcref{itm:direction_c2_2} and \labelcref{itm:direction_c2_3} hold. Furthermore, $(H_k)_{k \leq \ell + 1}$ can be an initial segment of a witness since $H_{\ell + 1}$ intersects $H_k$ for $k \leq \ell$ if and only if $k= \ell$. This completes the construction of the witness $(H_k)_{k \in \NN}$.

            It remains to prove that $d(Y)$ contains a tail of $N$ for every finite set $Y \subseteq S$. Let $Y \subseteq S$ be some finite set. Let $n \in \NN$ such that $Y \cap \bigcup_{k \geq n} H_k = \emptyset$ and $Y \cap K_n = \emptyset$. Since $m$ is strictly increasing, $m(n) \geq n$ holds. Thus $B_{m(n)} \subseteq K_{m(n)} \subseteq K_n$ holds. Therefore $\emptyset \neq B_{m(n)} \subseteq K_n \cap \bigcup_{k \geq n} H_k$ holds. Note that $K_n \subseteq d(Y)$ holds, since $Y \cap K_n = \emptyset$. This implies that the tail $\bigcup_{k \geq n} H_k$ of $N$ is contained in $d(Y)$, which finishes the proof. \qedhere
    \end{description}
\end{proof}

\section{Relation to ends of undirected and directed graphs}\label{sec:relation_ends}

In this section we show that the notion of ends in connectoids is indeed a generalisation of ends in undirected and directed graphs.
More precisely, we prove that the end spaces of undirected and directed graphs are homeomorphic to the end spaces of the corresponding connectoids under omitting limit edges for directed graphs.
Furthermore, we explain how edge-ends can be displayed by ends of connectoids.

\subsection{Ends of undirected graphs}
An \emph{end} of an undirected graph $G$ is an equivalence class of rays under the relation of having tails in a common component of $G - A$ for every finite vertex set $A \subseteq V(G)$~\cite{DiestelBook2016}.
For an end $\omega$ and a finite set $A \subseteq V(G)$ we refer to this component as $K(A, \omega)$. The \emph{end space} of $G$ is generated by the basic open sets of the form
$\{ \omega \in \Omega(G): K(A, \omega) = K\}$ for finite sets $A$ and components $K$ of $G - A$~\cite{DiestelBook2016}.
A \emph{direction} $d$ of an undirected graph $G$ is a map $d$ sending every finite set $A \subseteq V(G)$ to a component of $G - A$ such that $d(B) \subseteq d(A)$ holds for every $A \subseteq B \in V(G)^{< \infty}$~\cite{diestel2003graph}.

We use the direction theorem for undirected graphs:
\begin{thm}[\cite{diestel2003graph}*{Theorem~2.2}] \label{thm:direction_undirected}
    Let $G$ be an undirected graph. For every end $\omega$ of $G$ there is a unique direction $d_\omega$ such that $d_\omega(A) = K(A, \omega)$ for every finite set $A \subseteq V(G)$. The map sending $\omega$ to $d_\omega$ is a bijection between the ends of $G$ and the directions of $G$.
\end{thm}

For a ray $R$ in $G$ its vertex set $V(R)$ is a necklace in $(V(G), \cC)$ and we refer to it as the \emph{necklace corresponding to $R$}. Given two equivalent rays $R_1$ and $R_2$ and an arbitrary finite set $A \subseteq V(G)$, $R_1$ and $R_2$ have tails in the same component of $G - A$. Thus the $A$-tails of the corresponding necklaces $V(R_1)$ and $V(R_2)$ are contained in the same component in $\cK(V(G) \setminus A)$.
Therefore the necklaces $V(R_1)$ and $V(R_2)$ are equivalent.
This induces a well-defined, canonical map $\alpha$ sending ends of $G$ to ends of the corresponding connectoids.

\begin{lemma}\label{lem:end_spaces_homeo_undirected}
    Let $G$ be an undirected graph and let $(V(G), \cC)$ be the corresponding connectoid. Then the map $\alpha$ sending every end $\omega \in \Omega(G)$ to the end in $\Omega(V(G), \cC)$ that contains the necklaces corresponding to rays in $\omega$ is a homeomorphism between the end space of $G$ and the end space of $(V(G), \cC)$.
\end{lemma}
\begin{proof}
    There is a canonical bijection $\beta_1$ between the directions of $G$ and the directions of $(V(G), \cC)$:
    Given a direction $d$ of $G$ we obtain a direction of $(V(G), \cC)$ by mapping onto the vertex sets of the graph-theoretic components instead of the graph-theoretic components themself. Conversely, mapping to the induced subgraphs instead of the components transfers directions of $(V(G), \cC)$ to directions of $G$.

    Now concatenating $\beta_1$ with the bijection $\beta_0$ between ends and directions in undirected graphs from \cref{thm:direction_undirected} and the bijection $\beta_2$ between directions and ends in connectoids from \cref{thm:direction} gives a bijection between the ends of $G$ and the ends of $(V(G), \cC)$.
    
   We prove that the basic open sets of $\Omega(G)$ are mapped to basic open sets of $\Omega(V(G), \cC)$ and vice versa.
    Then $\beta_2 \circ \beta_1 \circ \beta_0$ is a homeomorphism between the end space of $G$ and the end space of $(V(G), \cC)$.
    Let $A \subseteq V(G)$ be an arbitrary finite set and let $K$ be some component of $G - A$.
    The function $\beta_1 \circ \beta_0$ maps the elements of the basic open set $O:= \{\omega \in \Omega(G): K(A, \omega) = K \}$ exactly to the directions of $(V(G), \cC)$ that send $A$ to $V(K)$. Thus $\beta_2 \circ \beta_1 \circ \beta_0$ maps the elements of $O$ exactly to the set $\{ \psi \in \Omega(V(G), \cC): K(\psi, A) = V(K) \}$.
    As every basic open set of $\Omega(G)$ is of the form $\{\omega \in \Omega(G): K(A, \omega) = K \}$ and every basic open set of $\Omega(V(G), \cC)$ is of the form $\{\omega \in \Omega(V(G), \cC): K(A, \omega) = V(K) \}$ for some finite set $A \subseteq V(G)$ and some component $K$ of $G - A$, we are done.
    
    Finally, we show that the bijection $\beta_2 \circ \beta_1 \circ \beta_0$ coincides with the map $\alpha$: Let $\omega \in \Omega(G)$ be arbitrary and $R$ be some ray in $\omega$. Then $\beta_0$ maps $\omega$ to the direction $d_\omega$ with $d_\omega(A) = K(A, \omega)$ for every finite set $A \subseteq V(G)$, i.e.\ $d_\omega(A)$ is the component of $G - A$ that contains a tail of $R$.
	Thus $\beta_1 \circ \beta_0$ maps $\omega$ to the direction $d_\omega'$ of $(V(G), \cC)$ such that $d_\omega'(A)$ contains a tail of the necklace $V(R)$ for every finite $A \subseteq V(G)$. By the definition of $\beta_2$, $K(A, \beta_2 \circ \beta_1 \circ \beta_0(\omega))$ contains a tail of $V(R)$ for every finite $A \subseteq V(R)$. Thus $V(R)$ is contained in $\beta_2 \circ \beta_1 \circ \beta_0(\omega)$, which proves $\alpha(\omega) = \beta_2 \circ \beta_1 \circ \beta_0 (\omega)$. Since $\omega$ was chosen arbitrarily, $\alpha = \beta_2 \circ \beta_1 \circ \beta_0$ holds.
\end{proof}

\subsection{Edge-ends of undirected graphs}
An \emph{edge-end} of an undirected graph $G$ is an equivalence class of rays under the relation of having tails in a common component of $G - F$ for every finite edge set $F \subseteq E(G)$~\cite{hahn1997edge}.

Given some ray $R$ in an undirected graph $G$, the edge-set $E(R)$ is a necklace in $(E(G), \cC_E)$, where $\cC_E:= \{C \subseteq E(G): G[C] \text{ is connected} \}$:
The family of connected sets $H_n$ consisting of the two edges incident with the $n+1$-st vertex of $R$ for every $n \in \NN$ is a witness of $E(R)$.

We show that the edge-ends of $G$ form ends of the connectoid $(E(G), \cC_E)$:

\begin{prop}\label{prop:relation_edge_ends}
	Let $G$ be an undirected graph and $\cC_E:= \{C \subseteq E(G): G[C] \text{ is connected} \}$.
	Then two rays $R_1$ and $R_2$ in $G$ are equivalent if and only if $E(R_1)$ and $E(R_2)$ are equivalent in $(E(G), \cC_E)$.
\end{prop}
\begin{proof}
	Note that for every $F \subseteq E(G)$ we have $\cK(E(G) \setminus F) = \{E(K): K \text{ component in } G - F \}$ by definition of $\cC_E$.
	
	Two rays $R_1$ and $R_2$ are equivalent if and only if for every finite set $F \subseteq E(G)$ there exist tails $R_1^F$ and $R_2^F$ of $R_1$ and $R_2$ respectively that are contained in some component $K$ of $G - F$.
	Given some finite set $F \subseteq E(G)$, there exist tails $R_1^F$ and $R_2^F$ of $R_1$ and $R_2$ respectively that are contained in some component $K$ of $G - F$ if and only if the tails $E(R_1^F)$ and $E(R_2^F)$ of $E(R_1)$ and $E(R_2)$ are contained in the component $E(K) \in \cK(E(G) \setminus F)$.
	We can deduce that $R_1$ and $R_2$ are equivalent if and only if $E(R_1)$ and $E(R_2)$ are equivalent in $(E(G), \cC_E)$.
\end{proof}
\noindent
Further, \cref{prop:relation_edge_ends} shows that the canonical map $\phi$ from the edge-ends of an undirected graph $G$ to the ends of $(E(G), \cC_E)$ is injective.
In general, $\phi$ is not surjective:
Let $G$ be an infinite star with $E(G):=\{e_n: n \in \NN\}$.
Then $G$ does not contain a ray and therefore no edge-end.
But the family $(\{e_n, e_{n + 1}\})_{n \in \NN}$ witnesses that $\{e_n: n \in \NN\}$ is a necklace in $(E(G), \cC_E)$ and thus $(E(G), \cC_E)$ contains some end.

\subsection{Ends of directed graphs}
We turn our attention to ends in directed graphs.
A \emph{directed ray} is a directed graph whose underlying undirected graph is a ray and whose edges are oriented away from the unique vertex of undirected degree $1$~\cite{burger2020ends}.
A \emph{tail} of a directed ray $R$ is a subgraph of $R$ such that its underlying undirected graph is a tail of the underlying undirected graph of $R$~\cite{burger2020ends}.
A directed ray in a directed graph $D$ is \emph{solid} if $R$ has a tail in a strong component of $D - A$ for every finite set $A \subseteq V(D)$~\cite{burger2020ends}. Two solid directed rays are equivalent if they have a tail in the same strong component of $D - A$ for every finite set $A \subseteq V(D)$~\cite{burger2020ends}. The equivalence classes of solid rays are the \emph{ends} of $D$~\cite{burger2020ends}.
For an end $\omega$ and a finite set $A \subseteq V(D)$ we let $K(A, \omega)$ be the strong component containing tails of the directed rays in $\omega$.
The topology of the \emph{end space} of $D$ is generated by the basic open sets of the form
$\{ \omega \in \Omega(D): K(A, \omega) = K\}$ for finite sets $A \subseteq V(D)$ and strong components $K$ of $D - A$~\cite{burger2020ends}.

A \emph{direction $d$} of a directed graph $D$ is a map $d$ sending every finite set $A \subseteq V(D)$ to a strong component of $D - A$ such that $d(B) \subseteq d(A)$ holds for every $A \subseteq B \in V(G)^{< \infty}$~\cite{burger2020ends}.
B\"urger and Melcher showed:
\begin{thm}[\cite{burger2020ends}*{Theorem~2}] \label{thm:direction_directed}
    Let $D$ be a directed graph. For every end $\omega$ of $D$ there is a unique direction $d_\omega$ such that $d_\omega(A) = K(A, \omega)$ for every finite set $A \subseteq V(D)$. The map sending $\omega$ to $d_\omega$ is a bijection between the ends of $D$ and the directions of $D$.   
\end{thm}
\noindent

Unlike in undirected graphs, two distinct ends $\omega$ and $\eta$ of a directed graph can contain solid rays $R_1 \in \omega$ and $R_2 \in \eta$ with infinitely many disjoint directed $R_1$--$R_2$~paths.
To capture this, B\"urger and Melcher introduced \emph{limit edges}:
there is a \emph{limit edge} starting in $\omega \in \Omega(D)$ and ending in $\eta \in \Omega(D)$ if there are $R_1 \in \omega$ and $R_2 \in \eta$ with infinitely many disjoint directed $R_1$--$R_2$~paths \cite{burger2020ends}.
In a similar way, they defined \emph{limit edges} between vertices and ends and proved that limit edges correspond one-to-one to edge-directions \cite{burger2020ends}*{Theorem~3}.
The connectoid corresponding to a directed graph $D$ does not capture the edge orientations of $D$ and thus cannot represent the limit edges of $D$.

\begin{lemma}\label{lem:end_spaces_homeo_directed}
    Let $D$ be a directed graph and let $(V(D), \cC)$ be the corresponding connectoid. Then there is a homeomorphism between the end space of $D$ and the end space of $(V(D), \cC)$.   
\end{lemma}

\begin{proof}
    Similarly as for undirected graphs, there is a canonical bijection between the directions of a directed graph and the directions of its corresponding connectoid. Together with the bijections of \cref{thm:direction_directed} and \cref{thm:direction} we obtain a bijection $\alpha$ between the ends of a directed graph and the ends of its corresponding connectoid.
    As in the proof of \cref{lem:end_spaces_homeo_undirected}, we can show that all basic open sets of $\Omega(D)$ are mapped to basic open sets of $\Omega(V(D), \cC)$ and vice versa.
    Thus $\alpha$ is the desired homeomorphism.
\end{proof}

Now we can view ends of directed graphs in terms of necklaces. A \emph{directed necklace} is a directed graph whose vertex set is a necklace in the corresponding connectoid.\footnote{B\"urger and Melcher introduced directed necklaces in \cites{burger2020ends, burger2020ends2, burger2020ends3} under the name \emph{necklaces} using an equivalent definition.}
Two directed necklaces are \emph{equivalent} if their vertex sets are equivalent as necklaces.
By \cref{lem:end_spaces_homeo_directed}, the equivalence classes of directed necklaces give an equivalent definition of ends in directed graphs.

Finally, we remark that our direction theorem for connectoids, \cref{thm:direction}, implies the direction theorems for undirected graphs, \cref{thm:direction_undirected}, and directed graphs, \cref{thm:direction_directed}.

\section{Normal trees}\label{sec:normal_tree}
An essential property of normal trees in undirected graphs is the one-to-one correspondence between the ends of a normal tree and the ends in the closure of its vertex set in its host graph~\cite{burger2022duality}*{Lemma~2.11}.
For weak normal trees there is in general no one-to-one correspondence between its ends and the ends in the closure of its vertex set:
Let $G$ be a countable star and let $T$ be a ray containing all leaves of $G$ but not the center vertex.
Then $T$ is a weak normal tree of the corresponding connectoid $(V(G), \cC)$.
As the tree $T$ contains an end but $G$ does not, there does not exist a one-to-one correspondence between the ends of $T$ and the ends of~$(V(G),\mathcal{C})$.

Hence we demand an additional property for the rooted rays of normal trees:
A weak normal tree $T$ of $(S, \cC)$ is a \emph{normal tree} of $(S, \cC)$ if for every rooted ray $R$ in $T$ there is a necklace in $(S, \cC)$ that contains almost all vertices of $R$.
\cref{prop:converging} characterises when this condition holds for $V(R)$.

Note that every weak normal tree $T$ that is spanning is a normal tree:
Let $R$ be some rooted ray of $T$ and $Y \subseteq S$ some finite set.
As $T$ is spanning, $Y \subseteq S = V(T)$ and thus there is $r \in V(R)$ such that $\Up{r}_T \cap Y = \emptyset$.
By \cref{prop:equivalence_weak_normal_tree}, $\Up{r}$ is a component of $\cK(S \setminus \ODown{r}_T)$ and in particular, contained in a component of $\cK(S \setminus Y)$.
Since almost all elements of $V(R)$ are contained in $\Up{r}_T$, property \labelcref{itm:converging_2} of \cref{prop:converging} is satisfied.
Then \cref{prop:converging} implies that $T$ is indeed a normal tree.

The existence of normal trees of undirected graphs and normal trees of their corresponding connectoids are equivalent in the following sense:
\begin{prop}\label{prop:existence_of_normal_tree}
    Let $G$ be an undirected graph, let $(V(G), \cC)$ be the corresponding connectoid and let $U \subseteq V(G)$ be some set.
    Then the following properties are equivalent:
    \begin{enumerate}[label=(\alph*)]
        \item\label{itm:existenc_of_normal_tree1} there is a normal tree in $G$ that contains $U$,
        \item\label{itm:existenc_of_normal_tree2} there is a normal tree of $(V(G), \cC)$ that contains $U$, and
        \item\label{itm:existenc_of_normal_tree3} there is a weak normal tree of $(V(G), \cC)$ that contains $U$.
    \end{enumerate}
\end{prop}
\begin{proof}
	\begin{description}
		\item[\labelcref{itm:existenc_of_normal_tree1} implies \labelcref{itm:existenc_of_normal_tree2}]
			Let $T$ be a normal tree in $G$. Then for every two $\leq_T$-incomparable vertices $u, v$ and every connected subgraph $H$ containing $u$ and $v$, $H$ contains a vertex of $\Down{u}_T \cap \Down{v}_T$.
		Every two $\leq_T$-comparable vertices $u \leq_T v$ are connected via the path $T[\Up{u}_T \cap \Down{v}_T]$.
		Furthermore, for every rooted ray $R$ in $T$ the set $V(R)$ forms a necklace and thus $V(R)$ converges to an end. 
		Therefore $T$ is a normal tree of $(V(G), \cC)$.
		\item[\labelcref{itm:existenc_of_normal_tree2} implies \labelcref{itm:existenc_of_normal_tree3}]
			   Every normal tree is a weak normal tree.
		\item[\labelcref{itm:existenc_of_normal_tree3} implies \labelcref{itm:existenc_of_normal_tree1}]
				    Given a weak normal tree $T$ of $(V(G), \cC)$ containing $U$, the tree $T$ is a subgraph of $G \cup T$.
				    The tree $T$ is a normal tree in $G \cup T$. 
				    We can assume without loss of generality that $G$ is connected (otherwise consider the induced subconnectoid on the component containing $U$).
				Since the existence of normal trees is closed under taking connected subgraphs~\cite{pitz2021proof}*{Theorem~1.2}, $G$ also contains a normal tree covering $U$. \qedhere
	\end{description}
\end{proof}
The proof of \cref{prop:existence_of_normal_tree} shows that normal trees of undirected graphs form normal trees in the corresponding connectoid.
But a normal tree of a connectoid that corresponds to an undirected graph $G$ does not form a normal tree in $G$ in general (see \cref{fig:normal_tree_difference}):
Let $G$ be a comb with $V(G):= \{s_n, \ell_n: n \in \NN\}$ and $E(G):=\{\{s_n, s_{n + 1}\}, \{s_n, \ell_n\}: n \in \NN\}$. Then the ray $R$ following the sequence $\ell_1, s_1, \ell_2, s_2, \ell_3, s_3, \dots$ is a normal tree of the corresponding connectoid $(V(G), \cC)$. But the ray $R$ is not a normal tree in $G$ since $\{s_1, \ell_2\} \notin E(G)$.

\begin{center}
	\begin{figure}[ht]
		\begin{tikzpicture}
			
			 \foreach \x in {0,1,2,3,4} \draw[line width=6pt,opacity=0.6,color=CornflowerBlue,bend right=0,line cap=round] (\x,0) to (\x,1);
			 
			 \foreach \x in {0,1,2,3} \draw[line width=6pt,opacity=1,color=white,bend right=0,line cap=round] (\x,0) to (\x+1,1);
			 
			 \foreach \x in {0,1,2,3} \draw[line width=6pt,opacity=0.6,color=CornflowerBlue,bend right=0,line cap=round] (\x,0) to (\x+1,1);
			 
			 \draw[line width=6pt, path fading=east, opacity=0.6,color=CornflowerBlue,bend right=0,line cap=round] (4,0) to (5,1);
			
			  \foreach \x in {0,1,2,3,4} \foreach \y in {0,1} \draw[fill] (\x,\y) circle [radius=.05];
			  
			  \foreach \x in {0,1,2,3} \draw[thick] (\x,0) to (\x+1,0);
			  
			  \foreach \x in {0,1,2,3,4} \draw[thick] (\x,0) to (\x,1);
			  
			  \draw (-0.3,1) node {$r$};
			  
			  \draw[path fading=east] (4,0) -- (5,0);
		\end{tikzpicture}
\caption{A \textcolor{CornflowerBlue}{normal tree} of the comb rooted at $r$ that is not a subgraph of the comb.}
\label{fig:normal_tree_difference}
	\end{figure}
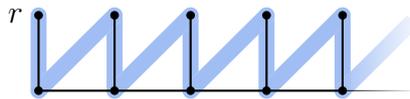
\end{center}

In the second paper of this series we show that for every connectoid (not just connectoids corresponding to undirected graphs) the existence of a normal tree covering a set $U$ is equivalent to the existence of a weak normal tree covering $U$.

\subsection{End-faithfulness of normal trees}
In this subsection we elaborate the correspondence between the ends of a connectoid and the ends of a normal spanning tree of this connectoid:

\EndFaithful*

We show a more general, local version of \cref{thm:endfaithfullness}.
We begin by introducing a topological space $|(S, \cC)|$ for a connectoid $(S, \cC)$ together with its ends similar to the space $|G|$ for an undirected graph $G$~\cite{diestel2004infinite}.
The topological space $|(S, \cC)|$ is defined on the set $S \cup \Omega(S, \cC)$ by the following basic open sets: for every $s \in S$ the set $\{s\}$ is open and for every end $\omega \in \Omega(S, \cC)$ and every finite set $X \subseteq S$ the set $K(X, \omega) \cup \Omega(X, \omega)$ is open.

Now, the \emph{closure} of a set $U \subseteq S \cup \Omega(S, \cC)$ is defined as in $|(S, \cC)|$, i.e.\ an end $\omega$ is in the closure of $U$ if $(K(X, \omega)\cup \Omega(X, \omega)) \cap U \neq \emptyset$ for every finite set $X \subseteq S$.

\begin{thm}\label{thm:endfaithfullness_local}
Let $(S, \cC)$ be a connectoid and $T$ a normal tree of $(S, \cC)$.
Then for every end $\omega \in \Omega(T)$ there exists an end $\eta(\omega) \in \Omega(S, \cC)$ to which the vertex set of the rooted ray $R_\omega \subseteq T$ converges.
The map $\eta$ sending $\omega$ to $\eta(\omega)$ is a bijection between the ends of $T$ and the ends in the closure of $V(T)$.
\end{thm}

\begin{proof}
    Let $\omega \in \Omega(T)$ be arbitrary. By the definition of normal tree, the set $V(R_\omega)$ converges to an end $\eta(\omega)$. Since $K(X, \eta(\omega))$ contains an element of $V(R_\omega) \subseteq V(T)$ for every finite set $X \subseteq S$, the end $\eta(\omega)$ is in the closure of $V(T)$. Thus $\eta$ maps into the closure of $V(T)$.

    Let $\omega' \in \Omega(T) \setminus \{ \omega \}$ be arbitrary.
    Note that the set $V(R_\omega \cap R_{\omega'})$ separates $V(R_\omega) \setminus V(R_\omega \cap R_{\omega'})$ and $V(R_{\omega'}) \setminus V(R_\omega \cap R_{\omega'})$ in $(S, \cC)$ since $T$ is normal.
    Thus the component $K(V(R_\omega \cap R_{\omega'}), \eta(\omega))$ containing almost all of $V(R_\omega)$ and the component $K(V(R_\omega \cap R_{\omega'}), \eta(\omega'))$ containing almost all of $V(R_{\omega'})$ are distinct.
    In particular, $\eta(\omega) \neq \eta(\omega')$. This shows that $\eta$ is injective.
    
    It remains to prove that $\eta$ is surjective. Let $\nu$ be some end in the closure of $V(T)$. We construct a rooted ray $(t_n)_{n \in \NN}$ in $T$ with the property $K(\ODown{t_n}_T, \nu) = K_{t_n}^T$. We set $t_1:= r$ and note that $K( \emptyset, \nu) = K_r^T$ since $\nu$ is in the closure of $V(T)$ and as $V(T) \subseteq K_r^T$.
    
    We assume that $t_n$ has been defined.
    As $\nu$ is an element in the closure of $V(T)$, the connected set $K(\Down{t_n}_T, \nu)$ contains some element of $V(T)$.
    Note that $K(\Down{t_n}_T, \nu) \cap V(T) \subseteq K(\ODown{t_n}_T, \nu) \cap V(T) = K_{t_n}^T \cap V(T) = \Up{t_n}$ holds.
    Thus there exists $x \in K(\Down{t_n}_T, \nu) \cap \OUp{t_n}$.
    Let $t_{n +1}$ be the child of $t_n$ with $x \in \Up{t_{n +1}}_T$.
    Then $K(\ODown{t_{n + 1}}_T, \nu) = K(\Down{t_n}_T, \nu) = K_{t_{n + 1}}^T$ since $x \in \Up{t_{n+1}}_T \subseteq K_{t_{n + 1}}^T$ by  \cref{prop:equivalence_weak_normal_tree}.
    This finishes the construction of the rooted ray $(t_n)_{n \in \NN}$.

    Let $\omega \in \Omega(T)$ be the end containing the ray $(t_n)_{n \in \NN}$. We show that $\eta(\omega)= \nu$. Then $\eta$ is surjective.
    Suppose for a contradiction that $\eta(\omega) \neq \nu$.
    There is a finite set $X \subseteq S$ such that $K(X, \eta(\omega)) \neq K(X, \nu)$.
    Note that the set $Y:= \{t_n: n \in \NN \} \setminus K(X, \eta(\omega))$ is finite since $\{t_n: n \in \NN \}$ converges to $\eta(\omega)$.
    Thus the component $K(X \cup Y, \nu)$ avoids $\{t_n: n \in \NN \}$ by the choice of $Y$ and since $K(X \cup Y, \nu) \subseteq K(X, \nu)$.
    We show that $K(X \cup Y, \nu)$ contains two elements $u, v \in T$ such that $\Down{u}_T \cap \Down{v}_T \subseteq \{t_n: n \in \NN \}$, contradicting that $K(X \cup Y, \nu)$ is connected as $T$ is normal.

	Since $\nu$ is in the closure of $V(T)$ there exists some $u \in K(X \cup Y, \nu) \cap V(T)$.
	Pick $m \in \NN$ such that
    $u \notin \Up{t_m}_T$.
    Let $v$ be some element of $K(X \cup Y \cup \ODown{t_m}, \nu) \cap V(T)$, which exists since $\nu$ is in the closure of $V(T)$.
    Note that $K(\ODown{t_m}_T, \nu) \cap V(T) = K_{t_m}^T \cap V(T) = \Up{t_m}_T$ by choice of $t_m$ and since $T$ is normal.
    This implies $v \in K(X \cup Y \cup \ODown{t_m}, \nu) \cap V(T) \subseteq K(\ODown{t_m}_T, \nu) \cap V(T) = \Up{t_m}_T$.
    Thus $\Down{u}_T \cap \Down{v}_T \subseteq \{t_n: n \in \NN \}$, which completes the proof.
\end{proof}

\section{Approximating connectoids by normal trees}\label{sec:approx}
In this section we show that normal trees can cover connectoids up to arbitrarily small open sets around their ends:

\Approximate*

\subsection{Extending normal trees}
We begin by investigating how normal trees can be extended.
Let $T$ be a weak normal tree of a connectoid $(S, \cC)$ and let $K$ be a component in $\cK(S \setminus V(T))$.
We call the set $N_K:= \{t \in T: K_t^T \supseteq K \}$ the \emph{neighbourhood} of $K$ in $T$.

\begin{prop}\label{prop:neighbourhood}
	Let $(S, \cC)$ be a connectoid, let $T$ be a weak normal tree of $(S, \cC)$ and let $K$ be a component in $\cK(S \setminus V(T))$.
	Then $K$ is a component in $\cK(S \setminus N_K)$ and $N_K$ is a branch.
\end{prop}
\begin{proof}
	Suppose for a contradiction that $K$ is not a component in $\cK(S \setminus N_K)$.
	Then $K$ is a proper subset of a component $K' \in \cK(S \setminus N_K)$.
	Thus $K'$ has to intersect $V(T) \setminus N_K$.
	Let $t \in (V(T) \setminus N_K) \cap K'$ be $\leq_T$-minimal.
	Then $K' \cap \ODown{t}_T = \emptyset$ and thus $K'$ is subset of $K_t^T$.
	This implies $K \subset K' \subseteq K_t^T$ and therefore $t \in N_K$, a contradiction.
	
	For $t \in N_K$ and $s \in \ODown{t}_T$ the inclusions $K_s^T \supseteq K_t^T \supseteq K$ hold, which implies $s \in N_K$.
	Thus $N_K$ is $\leq_T$-down-closed.
	
	For two $\leq_T$-incomparable elements $s, t$, there is no connected set containing $s$ and $t$ that avoids $\Down{s}_T \cap \Down{t}_T$ by normality of $T$.
	This implies that $K_s^T$ and $K_t^T$ are disjoint and in particular, at most one of $s$ and $t$ is contained in $N_K$.
	Thus the elements of $N_K$ are pairwise $\leq_T$-comparable.
	This shows that $N_K$ is a branch.
\end{proof}

Let $(S, \cC)$ be a connected connectoid and let $\hat S \subseteq S$ be some subset.
A component $K \in \cK(S \setminus \hat S)$ has \emph{finite adhesion} to $\hat S$ if there exists a finite set $X \subseteq \hat S$ such that $K$ is a component in $\cK(S \setminus X)$.
Otherwise, the component $K$ has \emph{infinite adhesion} to $\hat S$.

\begin{prop}\label{prop:finite_neighbourhood}
	Let $(S, \cC)$ be a connected connectoid and $T$ a weak normal tree of $(S, \cC)$.
	Then a component $K \in \cK(S \setminus V(T))$ has finite adhesion to $V(T)$ if and only if $N_K$ is finite.
\end{prop}
\begin{proof}
	Suppose for a contradiction that $K \in \cK(S \setminus V(T))$ has finite adhesion and $N_K$ is infinite.
	Let $X \subseteq V(T)$ be a finite subset of $V(T)$ such that $K \in \cK(S \setminus X)$ and pick $t \in N_K \setminus \Down{X}_T$.
	Then $K_t^T$ is contained in a component in $\cK(S \setminus X)$ since $K_t^T \cap X \subseteq K_t^T \cap V(T) \cap X = \Up{t}_T \cap X = \emptyset$.
	That component must be $K$ since $t \in N_K$, so $t \in N_K$, contradicting $K \cap V(T) = \emptyset$.
	
	If $N_K$ is finite, $K$ has finite adhesion since $K \in \cK(S \setminus N_K)$ by \cref{prop:neighbourhood}.
\end{proof}

\begin{prop} \label{prop:extension_normal_tree}
	Let $(S, \cC)$ be a connected connectoid and let $T$ be a normal tree of $(S, \cC)$ rooted at $r$.
	Let $T_K$ be a (possibly trivial) normal tree of the induced subconnectoid on $K$ for every component $K \in  \cK(S \setminus V(T))$.
	Suppose that every component in $\cK(S \setminus V(T))$ has finite adhesion to $V(T)$.
	Then there exists a normal tree $T'$ of $(S, \cC)$ rooted at $r$ with $T \subseteq T'$ such that $T' - T$ is the disjoint union of the trees $T_K$ for $K \in \cK(S \setminus V(T))$.
	
	Furthermore, if $T$ and all $T_K$ for $K \in \cK(S \setminus V(T))$ are rayless, then $T'$ is rayless.
\end{prop}

\begin{proof}
	Let $K \in \cK(S \setminus V(T))$ with $T_K \neq \emptyset$.
	By \cref{prop:neighbourhood} and \cref{prop:finite_neighbourhood}, the elements of $N_K$ form a finite rooted path in $T$ and $K$ is a component in $\cK(S \setminus N_K)$.
	Let $t_K$ be the $\leq_T$-maximal element of $N_K$.
	
	Let $T'$ be the tree obtained from $T$ by placing for every such $K \in \cK(S \setminus V(T))$ the tree $T_K$ on top of $t_K$.
	We show that $T'$ is a weak normal tree and deduce that $T'$ is a normal tree.
	
	Firstly, let $C$ be some connected set and let $u_1, u_2 \in T'$ be two $\leq_{T'}$-incomparable elements in $C$.
	If $C \subseteq K$ for some component $K \in \cK(S \setminus V(T))$, then $u_1, u_2 \in K$ and $C$ contains an element of $\Down{u_1}_{T_K} \cap \Down{u_2}_{T_K} \subseteq \Down{u_1}_{T'} \cap \Down{u_2}_{T'}$ since $T_K$ is normal in the induced subconnectoid on $K$.
	Thus we can assume that $C$ is not contained in a component of $\cK(S \setminus V(T))$.
	If $u_i \in T' - T$, then there is a component $K_i \in \cK(S \setminus V(T))$ containing $u_i$ for $i \in \{1,2\}$.
	Then $C$ contains an element $u_i' \in N_{K_i} = \Down{u_i}_{T'} \cap V(T)$ since $K \in \cK(S \setminus N_{K_i})$ and as $C \not\subseteq K_i$.
	If $u_i \in T$, set $u_i':= u_i$ for $i \in \{1,2\}$.
	Since $T$ is normal, $C$ contains an element of $\Down{u_1'} _T \cap \Down{u_2'}_T  \subseteq \Down{u_1}_{T'} \cap \Down{u_2}_{T'}$.
	
	Secondly, let $v, w \in T'$ be arbitrary with $v \leq_{T'} w$.
	If $v \in T' - T$, then $v, w$ are contained in a component $K \in \cK(S \setminus V(T))$ and in particular, $v \leq_{T_K} w$.
	Since $T_K$ is normal in $K$, there is a connected set $C \subseteq K$ avoiding $\ODown{v}_{T_K}$ that contains $v$ and $w$.
	Then $C$ avoids $\ODown{v}_{T'}$ since $\ODown{v}_{T'} \cap K = \ODown{v}_{T_K}$.
	
	If $v \in T$, we consider the connected set $K_v^T$, which avoids $\ODown{v}_T = \ODown{v}_{T'}$.
	By \cref{prop:equivalence_weak_normal_tree}, $K_v^T$ contains $\Up{v}_T$.
	By construction, every element of $\Up{v}_{T'} \setminus \Up{v}_T$ is contained in a component $K \in \cK(S \setminus V(T))$ with $v \in N_K$.
	The property $v \in N_K$ implies that $K \subseteq K_v^T$, which proves that $\Up{v}_{T'} \setminus \Up{v}_T$ is contained in $K_v^T$.
	This shows that every $w \geq_{T'} v$ is contained in the connected set $K_v^T$ and shows that $T'$ is a weak normal tree.
	
	Finally, let $R$ be an arbitrary rooted ray in $T'$.
	The ray $R$ has a tail $R'$ that is either contained in $T$ or contained in some $T_K$, by construction.
	As $T$ and all $T_K$ are normal, there exists a necklace $N$ of $(S, \cC)$ that contains almost all vertices of $R'$.
	In particular, $N$ contains almost all vertices of $R$ and thus $R$ converges to an end.
	This proves that $T'$ is a normal tree.
\end{proof}

\begin{cor}\label{cor:finite_normal_tree}
	Let $(S, \cC)$ be a connectoid, let $s \in S$ be some element and $X \subseteq S$ a finite set.
	Then there exists a finite normal tree of $(S, \cC)$ rooted at $s$ with $V(T) = X \cup \{s\}$.
\end{cor}
\begin{proof}
	Enumerate $X:= \Set{x_1, \dots, x_n}$ and set $T_0 := \{s\}$.
	We construct an increasing sequence $(T_m)_{m \in [n]}$ of normal trees with $V(T_m) := \Set{s, x_1, \dots, x_m}$.
	Then $T_n$ is as desired.
	If $T_{m - 1}$ has been constructed, let $T_m$ be the normal tree obtained by applying \cref{prop:extension_normal_tree} to $T_{m - 1}$ and the normal tree $\{x_m\}$.
\end{proof}

\subsection{Proof of \cref{thm:approx}}
The proof of \cref{thm:approx} is based on Kurkofka, Melcher and Pitz' proof of its graph-theoretic counterpart \cite{kurkofka2021approximating}*{Theorem~1}.
As a preparation for this proof we show how ends of subconnectoids relate to the ends of their host connectoid.

\begin{prop}\label{prop:ends_of_subconnectoid}
	Let $(S, \cC)$ be a connectoid and $(S', \cC')$ a subconnectoid of $(S, \cC)$. Then for every end $\omega' \in \Omega(S', \cC')$ there exists an end $\omega \in \Omega(S, \cC)$ such that $\omega' \subseteq \omega$.
	In particular, $K(X \cap S', \omega') \subseteq K(X, \omega)$ holds for every finite set $X \subseteq S$.
\end{prop}
\begin{proof}
	Every two equivalent necklaces in $(S', \cC')$ are also equivalent in $(S, \cC)$. Therefore every equivalence class $\omega'$ of necklaces in $(S', \cC)$ is contained in an equivalence class $\omega$ of necklaces in $(S, \cC)$.
	Let $X \subseteq S$ be an arbitrary finite set. The connected set $K(X \cap S', \omega') \in \cC'$ avoids $X$ and is therefore contained in a component in $\cK(S \setminus X)$.
	Since $\omega' \subseteq \omega$, $K(X \cap S', \omega') \subseteq K(X, \omega)$ holds.
\end{proof}

\begin{proof}[Proof of \cref{thm:approx}]
Let $\mathcal{K}= \{K(X_\omega, \omega): \omega \in \Omega(S, \cC)\}$ be an arbitrary collection. We call a set $S' \subseteq S$ \emph{bounded} if $S' \subseteq K(X_\omega, \omega)$ for some $\omega \in \Omega(S, \cC)$ and \emph{unbounded} otherwise.

We construct an increasing sequence $(T_n)_{n \in \NN}$ of rayless normal trees of $(S, \cC)$ with common root $r$. Set $T_1:= \{r\}$.
If $T_n$ has been defined for some $n \in \NN$, then we construct $T_{n +1}$ by extending $T_n$ into every unbounded component in $\cK(S \setminus V(T_n))$ finitely.

First of all, we prove that for every unbounded component $K$ in $\cK(S \setminus V(T_n))$ there is a finite set $S_K \subseteq K$ such that $\cK(K \setminus S_K)$ has either at least two unbounded connected sets or else none at all.
Suppose for a contradiction that there is $K \in \cK(S \setminus V(T_n))$ such that exactly one component in $\cK(K \setminus X)$ is unbounded for every finite set $X \subseteq K$. This defines a direction $d$ in the induced subconnectoid $(K, \cC^K)$ on $K$ and by \cref{thm:direction} an end $\omega \in \Omega(K, \cC^K)$.
By \cref{prop:ends_of_subconnectoid}, there exists an end $\omega' \in \Omega(S, \cC)$ such that $\omega \subseteq \omega'$.
Then $d(X_{\omega'} \cap K) = K(X_{\omega'} \cap K, \omega) \subseteq K(X_{\omega'}, \omega')\in \cK$ is bounded, contradicting the definition of $d$.

We define a finite normal tree $T_K$ of $(K, \cC^K)$ as follows:
If there is a finite set $S_K \subseteq K$ such that every component in $\cK(K \setminus S_K)$ is bounded, let $T_K$ be a finite normal tree of the induced subconnectoid $(K, \cK^K)$ on $K$ that contains $S_K$, which exists by \cref{cor:finite_normal_tree}.

Otherwise let $S_K \subseteq K$ be such that $\cK(K \setminus S_K)$ contains at least two unbounded components, which we call $C_K^1$ and $C_K^2$. Let $x_K \in K$ be an arbitrary element. Pick for each element $x$ in the neighbourhood $N_K$ of $K$ in $T_n$ a finite connected set $H_x \subseteq K_x^{T_n}$ that contains $x, x_K$, some element of $C_K^1$ and some element of $C_K^2$.
By \cref{cor:finite_normal_tree}, there is a finite normal tree $T_K$ of $K$ rooted at $x_K$ that contains $S_K \cup \bigcup_{x \in N_K} (H_x \cap K)$.

Every component in $\cK(S \setminus V(T_n))$ has finite adhesion to $V(T_n)$ by \cref{prop:neighbourhood} and \cref{prop:finite_neighbourhood} since $T_n$ is rayless.
Now let $T_{n + 1}$ be the rayless normal tree obtained by applying \cref{prop:extension_normal_tree} to $T_n$ and $T_K$ for all unbounded components $K$ in $\cK(S \setminus V(T_n))$. Finally, we set $T:= \bigcup_{n \in \NN} T_n$.

We show that $T$ is rayless. Suppose for a contradiction that there is a ray $R$ in $T$. Then the ray $R$ is built by extending $T_n$ into an unbounded component $K(n) \in \cK(S \setminus V(T_n))$ for every $n \in \NN$ such that $K(1) \supset K(2) \supset K(3) \supset \dots$, by construction.

Let $n \in \NN$ be arbitrary.
As $K(n+1) \subseteq K(n)$ and $K(n+1)$ avoids $S_{K(n)}$ since $S_{K(n)} \subseteq V(T_{n+1})$, there is a component in $\cK(K(n) \setminus S_{K(n)})$ containing $K(n + 1)$.
The component in $\cK(K(n) \setminus S_{K(n)})$ containing $K(n + 1)$ is unbounded as $K(n + 1)$ is unbounded.
Then there are at least two unbounded components $C_{K(n)}^1, C_{K(n)}^2$ in $\cK(K(n) \setminus S_{K(n)})$ by choice of $S_{K(n)}$.
Thus the tree $T_{K(n)}$ was constructed in the latter way for every $n \in \NN$.
Note that $(x_{K(n)})_{n \in \NN}$ is a strictly $\leq_T$-increasing sequence of vertices in $V(R)$ since $x_{K(n)}$ is the root of $T_{K(n)}$.

By construction, for every $n \in \NN$ there is a connected set $H_n \subseteq K(n)$ containing $x_{K(n)}, x_{K({n + 1})}$, some vertex of $C_{K(n+1)}^{1}$ and some vertex of $C_{K(n+1)}^{2}$ such that $H_n \cap K(n + 1) \subseteq V(T_{n + 2})$.
Thus $H_n \subseteq K(n) \setminus K(n + 2)$ holds since $K(n + 2) \cap V(T_{n + 2}) = \emptyset$.
Therefore $(H_n)_{n \in \NN}$ is the witness of some necklace $N$.
Let $\omega$ be the end containing $N$.

Pick $m \in \NN$ such that the tail $\bigcup_{n \geq m} H_n$ of $N$ and the set $K(m +1 ) \setminus K(m + 2)$ are disjoint to the finite set $X_\omega$.
Let $\epsilon \in \{1,2\}$ such that $C_{K(m+1)}^{\epsilon} \in \cK(K(m+1) \setminus S_{K(m+1)})$ is disjoint to $K(m+2)$, which is possible since $K(m+2)$ is contained in a component in $ \cK(K(m+1) \setminus S_{K(m+1)})$.
This implies that $C_{K(m+1)}^\epsilon \subseteq K(m+1) \setminus K(m+2)$ is disjoint to $X_\omega$. 
By construction, $H_m$ contains an element of $C_{K(m+1)}^\epsilon$ and therefore $C_{K(m+1)}^\epsilon$ is contained in the component $K(X_\omega, \omega) \in \cK(S \setminus X_\omega)$ containing $\bigcup_{n \geq m} H_n$.
Thus $C_{K(m+1)}^\epsilon$ is bounded, a contradiction.
Therefore $T$ is indeed rayless.

It remains to prove that every component in $\cK(S \setminus V(T))$ is bounded. Let $K \in \cK(S \setminus V(T))$ be arbitrary. By \cref{prop:neighbourhood}, the elements of the neighbourhood $N_K$ form a rooted path in $T$. Thus there is $n \in \NN$ such that $K$ is a component in $\cK(S \setminus V(T_n))$. As $T_{n + 1}$ does not contain an element of $K$, the component $K$ is bounded. This finishes the proof.
\end{proof}

\subsection{Ultra-paracompactness}
We deduce from \cref{thm:approx} that the end spaces of connectoids are \emph{ultra-paracompact}, i.e.\ every open cover of an end space can be refined to an open partition. Further, we show that every open subset of an end space is the end space of some connectoid. This implies that every open subset of an end space is ultra-paracompact.

Given a connectoid $(S, \cC)$, a finite set $X \subseteq S$ and some end $\omega \in \Omega(S, \cC)$, we define $\Omega(X, \omega):=  \{ \psi \in \Omega(S, \cC): K(X, \psi) = K(X, \omega) \}$.
\begin{prop}
    Let $(S, \cC)$ be a connected connectoid. For every open cover $U$ of $\Omega(S, \cC)$ there is an open partition of $\Omega(S, \cC)$ that refines $U$.
\end{prop}
\begin{proof}
    We can assume without loss of generality that there is a collection $(X_\omega)_{\omega \in \Omega(S, \cC)}$ of finite subsets of $S$ such that $U = \Set{\Omega(X_\omega, \omega): \omega \in \Omega(S, \cC)}$.
    By \cref{thm:approx} there is a rayless normal tree $T$ of $(S, \cC)$ such that every component in $\cK(S \setminus V(T))$ is contained in some $K(X_\omega, \omega)$.
    Since $T$ is rayless every component $K \in \cK(S \setminus V(T))$ has finite neighbourhood $N_K$ in $T$ by \cref{prop:neighbourhood}.
    Thus the set $\Omega(K):= \{ \omega \in \Omega(S, \cC): K(N_K, \omega) = K \}$ is open.
    Therefore $\{ \Omega(K): K \in \cK(S \setminus V(T))\}$ is an open partition of $\Omega(S, \cC)$.
\end{proof}

Akin to its counterpart for undirected graphs \cite{kurkofka2021approximating}*{Lemma~5.1}, we prove:
\begin{lemma}
    Open subsets of end spaces are again end spaces.
\end{lemma}

\begin{proof}
    Let $(S, \cC)$ be a connectoid and $\Gamma \subseteq \Omega(S, \cC)$ be an open subset. We pick a maximal family $\mathcal{N}$ of disjoint necklaces of $\Omega(S, \cC) \setminus \Gamma$, which exists by Zorn's lemma. Set $S':= S \setminus \bigcup \mathcal{N}$ and consider the induced subconnectoid $(S', \cC')$ on $S'$.
    By \cref{prop:ends_of_subconnectoid}, every end $\omega \in \Omega(S', \cC')$ is contained in an end of $(S, \cC)$ and by maximality of $\mathcal{N}$, this end is contained in $\Gamma$.
    We prove that the map $\phi$ sending every $\omega \in \Omega(S', \cC')$ to the unique end of $\Gamma$ containing $\omega$ is a homeomorphism between $\Omega(S', \cC')$ and $\Gamma$.

    Let $\gamma \in \Gamma$ be arbitrary and consider an arbitrary open set $\Omega(Y, \gamma) \subseteq \Gamma$ for some finite set $Y \subseteq S$. We show that $\bigcup \mathcal{N}$ intersects $K(Y, \gamma)$ finitely. As $\Omega(Y, \gamma) \subseteq \Gamma$, no necklace of $\mathcal{N}$ has a tail in $K(Y, \gamma)$. Thus every necklace of $\mathcal{N}$ has finite intersection with $K(Y, \gamma)$. Furthermore, every necklace of $\mathcal{N}$ that intersects $K(Y, \gamma)$ contains an element of $Y$ and therefore there are just finitely many such necklaces. This implies that $K(Y, \gamma) \cap \bigcup \mathcal{N}$ is finite.
    Then $\hat{Y}:= (K(Y, \gamma) \cap \bigcup \mathcal{N}) \cup Y$ is a finite set with the property $K( \hat{Y}, \gamma) \subseteq S'$.
    
    Let $N$ be a necklace of $\gamma$ that avoids $\hat{Y}$. Then $N$ is a necklace in $S'$ and there is $\omega \in \Omega(S', \cC')$ that contains $N$. Thus $\phi(\omega) = \gamma$ holds, which shows that $\phi$ is surjective.    
    Furthermore, $K( \hat{Y}, \gamma)$ is a component in $\cK(S' \setminus (\hat{Y} \cap S'))$ and therefore $\phi^{-1}(\Omega(\hat{Y}, \gamma)) = \Omega(\hat{Y} \cap S', \omega)$ is open. Thus $\phi$ is continuous since $Y$ was chosen arbitrarily.

    Let $\omega, \omega' \in \Omega(S', \cC')$ be ends with $\phi(\omega)= \phi(\omega') = \gamma$. As explained before, there is a finite set $\hat{Y} \subseteq S$ such that $K(\hat{Y}, \gamma) \subseteq S'$. Let $N$ be a necklace of $\omega$ that avoids $\hat{Y}$ and let $N'$ be a necklace of $\omega'$ that avoids $\hat{Y}$. Then $N, N'$ are necklaces of $\gamma$ that are contained in $K(\hat{Y}, \gamma) \subseteq S'$. By \cref{prop:necklace_property}, there are infinitely many disjoint finite connected sets that intersect $N$ and $N'$. Since $\hat{Y}$ is finite, finitely many of these connected sets intersect $\hat{Y}$. Thus almost all of these connected sets are contained in $K(\hat{Y}, \gamma) \subseteq S'$. Therefore $\omega, \omega'$ cannot be finitely separated in $S'$, which implies that $\omega = \omega'$ holds and shows that $\phi$ is injective.

    Finally, let $X \subseteq S'$ be arbitrary.
    Then the open set $\Omega(X \cap S', \gamma)$ gets mapped to the open set $\Omega(X, \gamma)$, which proves that $\phi$ is open.
\end{proof}

\begin{cor}\
    Open subsets of end spaces of connectoids are ultra-paracompact.
\end{cor}

\section{Characterisation of compactness}\label{sec:compactification}
We characterise the connectoids $(S, \cC)$ whose topological space $|(S, \cC)|$ is compact.
We recall that the topological space $|(S, \cC)|$ is defined on the set $S \cup \Omega(S, \cC)$ by the following basic open sets: for every $s \in S$ the set $\{s\}$ is open and for every end $\omega \in \Omega(S, \cC)$ and every finite set $X \subseteq S$ the set $K(X, \omega) \cup \Omega(X, \omega)$ is open.

We call a connectoid $(S, \cC)$ \emph{solid} if there are finitely many components in $\cK(S \setminus X)$ for every finite set $X \subseteq S$.
\begin{lemma}
	Let $(S, \cC)$ be a connectoid. Then $|(S, \cC)|$ is compact if and only if $(S, \cC)$ is solid.
\end{lemma}
\begin{proof}
	If $|(S, \cC)|$ is compact, let $X \subseteq S$ be an arbitrary finite set. The open partition $\cK(S \setminus X) \cup \{ \{x\}: x \in X  \}$ shows that $\cK(S \setminus X)$ is finite. Thus $(S, \cC)$ is solid.
	
	If $(S, \cC)$ is solid, we show that every (weak) normal tree of $(S, \cC)$ is locally finite: Let $T$ be some weak normal tree of $(S, \cC)$ and $t \in T$ arbitrary.
	By the definition of weak normal trees, no two children of $t$ are contained in the same component in $\cK(S \setminus \Down{t}_T)$. Since $\cK(S \setminus \Down{t}_T)$ is finite, $t$ has finitely many children.
	
	Now let $U$ be an arbitrary open cover of $|(S, \cC)|$. For every end $\omega \in \Omega(S, \cC)$ pick a finite set $X_\omega \subseteq S$ such that $K(X_\omega, \omega) \cup \Omega(X_\omega, \omega)$ is contained in an element of $U$.
	We apply \cref{thm:approx} to the collection $\{K(X_\omega, \omega) : \omega \in \Omega(S, \cC))\}$ to obtain a rayless normal tree $T$ such that every component in $\cK(S \setminus V(T))$ is contained in an element of the collection.
	By choice of this collection, each component in $\cK(S \setminus V(T))$ is subset of some element in $U$.
	
	By our observation, $T$ is locally finite. Thus $V(T)$ is finite and in particular $\cK(S \setminus V(T))$ is finite.
	We construct a finite subset $U' \subseteq U$ by picking for each component $K \in \cK(S \setminus V(T))$ an element of $U$ that contains $K$ and for every $t \in T$ an element of $U$ that contains $t$.
	Note that $U'$ covers $|(S, \cC)|$ by construction, which shows that $U'$ is as desired.
\end{proof}

Finally we remark that, akin to locally finite undirected graphs (see e.g.\ \cite{ouborny2023universal}), the topological space $|(S, \cC)|$ of every solid connectoid $(S, \cC)$ with countable ground set $S$ can be represented as an inverse limit of finite connectoids.

\section{Open problems}\label{sec:open_problems}
The end space of a connectoid with a normal spanning tree $T$ is homeomorphic to the end space of the undirected graph $T$, by \cref{thm:endfaithfullness}.
\begin{problem}
	Does there exist for every connectoid $(S, \mathcal{C})$ an undirected graph $G$ such that $\Omega(S, \mathcal{C})$ and $\Omega(G)$ are homeomorphic?
\end{problem}

The star-comb lemma is a fundamental tool for the investigation of infinite undirected graphs~\cite{DiestelBook2016}.
It proves the existence of a certain connected substructure, i.e.\ either a star or a comb, in every infinite connected undirected graph~\cite{DiestelBook2016}*{Lemma~8.2.2}.
We raise the question if the star-comb lemma can be generalised to connectoids:
\begin{problem}
	Is there a finite list $\mathcal{L}$ of coarse connected structures 
	such that for every connectoid $(S, \cC)$ and every infinite set $S' \subseteq S$ an element of $\mathcal{L}$ is a substructure of $(S, \cC)$ containing infinitely many elements of $S'$?
\end{problem}
\noindent
It seems plausible to demand that necklaces form an element of $\mathcal{L}$ as rays do in the star-comb lemma.

\section*{Acknowledgements}

The second author gratefully acknowledges support by a doctoral scholarship of the Studienstiftung des deutschen Volkes.

\bibliography{ref.bib}

\end{document}